\documentclass[a4paper,11pt,leqno]{amsart}

\usepackage{amsmath,amsfonts,amssymb}
\usepackage{pifont,graphicx,epsfig}
\usepackage[bf]{caption2}
\usepackage{epstopdf}
\usepackage{color}

\def\R{\hbox{\font\dubl=msbm10 scaled 1100 {\dubl R}}}

\def\N{\hbox{\font\dubl=msbm10 scaled 1100 {\dubl N}}}

\def\Re{{\rm{Re}}\,}

\usepackage{amssymb, amsmath}

\newtheorem{Theorem}{Theorem}
\newtheorem{Corollary}[Theorem]{Corollary}
\newtheorem{Proposition}[Theorem]{Proposition}
\newtheorem{Lemma}[Theorem]{Lemma}

\title[Negative values]{Negative values of the Riemann zeta function on the critical line}

\author[Justas Kalpokas, Maxim A. Korolev, J\"orn Steuding]{Justas Kalpokas, Maxim A. Korolev, J\"orn Steuding}

\thanks{The first author is supported by grant No MIP-94 from the
 Research Council of Lithuania, as well he would like to say thank you to his friend Gintautas Aidietis for support.}

\date{\today}

\begin{document}

\begin{abstract}
We investigate the intersections of the curve $\mathbb{R}\ni t\mapsto \zeta({1\over 2}+it)$ with the real axis. We show unconditionally that the zeta-function takes arbitrarily large positive and negative values on the critical line.
\end{abstract}

\maketitle

{\small \noindent {\sc Keywords:} Riemann zeta-function, value-distribution, critical line}
\\ \bigskip

%\tableofcontents

\section{Introduction and statement of the main results}

We investigate the value-distribution of the Riemann zeta-function $\zeta(s)$ on the critical line $s={1\over 2}+i\R$. Recently Soundararajan \cite{sou} succeeded to prove 
\begin{equation}\label{oomega}
\max_{t\in[T,2T]}\left|\zeta({\textstyle{\frac12}}+it)\right|\gg\exp\left((1+o(1))\sqrt{\frac{\log T}{\log\log T}}\right)\mbox{\quad as\quad  } T\rightarrow \infty.
\end{equation}
This result as well as all other $\Omega$-estimates for the Riemann zeta function does not provide special information about the location of the large values of the Riemann zeta function in the complex plane. 
 For this purpose we study the set of intersection points of the curve $t\mapsto \zeta({1\over 2}+it)$ with straight lines $e^{i\phi}\R$ through the origin and prove upper and lower bounds for associated discrete moments. In particular, we prove that $\zeta({1\over 2}+it)$ takes arbitrarily large positive values as well as arbitrarily large negative values. 
\smallskip

Recall the functional equation for the zeta-function,
\begin{equation}\label{feq}
\zeta(s)=\Delta(s)\zeta(1-s),\qquad \mbox{where}\quad \Delta(s):=2^s\pi^{s-1}\Gamma(1-s)\sin({\textstyle{\pi s\over 2}}).
\end{equation}
It follows immediately that $\Delta(s)\Delta(1-s)=1$, hence $\Delta({1\over 2}+it)$ lies on the unit circle for real $t$.  Given
an angle $\phi\in[0,\pi)$, denote by $t_n(\phi)$ with $n\in\N$ the positive roots of the equation
$$
e^{2i\phi}=\Delta({\textstyle{1\over 2}}+it)
$$
in ascending order. These roots correspond to intersections of the curve $t\mapsto \zeta({1\over 2}+it)$ with straight lines $e^{i\phi}\R$ through the origin (see Kalpokas and Steuding \cite{Kalp}). Of special interest are intersections with the real line; in this case $\phi=0$ and the roots are called Gram points (after Gram \cite{gram} who observed that the first of those roots separate consecutive zeta zeros on the critical line). 
\par

Building on work of Rudnick and Soundararajan \cite{Sound}, resp. Milinovich and Ng \cite{Mili} we shall establish a lower bound of the expected order for those discrete moments with arbitrary rational exponents:

\begin{Theorem}\label{Th1}
For any rational $k\geqslant 1$ and any $\phi\in[0,\pi)$, as $T\to\infty$,
$$
\sum_{0<t_n(\phi)\leqslant T}\left|\zeta\left({\textstyle{\frac12}}+it_n(\phi)\right)\right|^{2k}\gg T(\log T)^{k^2+1}.
$$
\end{Theorem}

\noindent Continuing recent work of Kalpokas and Steuding \cite{Kalp}, we shall derive an asymptotic formula for the third discrete moment:

\begin{Theorem}\label{Th2}
For any $\phi\in[0,\pi)$, as $T\to\infty$,
\begin{eqnarray*}
\lefteqn{\sum_{0<t_n^\phi\leqslant T}\zeta\left({\textstyle{\frac12}}+it_n(\phi)\right)^3}\\
&=&2e^{3i\phi}(\cos\phi)\,\frac{T}{2\pi}P_{3}\left(\log\frac{T}{2\pi}\right)+2e^{3i\phi}(\cos{3\phi})\,\frac{T}{2\pi}\log\frac{T}{2\pi
e}+O\left(T^{{1\over 2}+\varepsilon}\right),
\end{eqnarray*}
where $P_{3}$ is a computable polynomial of degree three.
\end{Theorem}

Combining the above theorems we shall deduce that there are arbitrary large positive and negative values of the Riemann zeta-function on the critical line. More generally, we shall show that the corresponding statement holds with respect to any direction $\phi$. Recall that $e^{-i\phi}\zeta(\frac12+it_n(\phi))$ is real. Hence, we may write $t_n^+(\phi)$ in place of $t_n(\phi)$ if $e^{-i\phi}\zeta(\frac12+it_n(\phi))\geqslant 0$ and  $t_n^-(\phi)$ if $e^{-i\phi}\zeta(\frac12+it_n(\phi))<0$.

\begin{Corollary}\label{Cor1}
For any $\phi\in[0,\pi)$, there are arbitrary large positive and negative values of $e^{-i\phi}\zeta({1\over 2}+it_n(\phi))$. More precisely,
$$
\max_{0<t_n^{\pm}(\phi)\leqslant T}\left|\zeta(\textstyle{\frac12}+it_n^{\pm}(\phi))\right|\gg (\log T)^{\frac54}.
$$
If Riemann hypothesis is assumed for any arbitrary small $\delta >0$ we have 
 $$
\max_{0<t_n^{\pm}(\phi)\leqslant T}\left|\zeta(\textstyle{\frac12}+it_n^{\pm}(\phi))\right|\gg (\log T)^{\frac32-\delta}.
$$

\end{Corollary}

If we are interested only in extreme values on a given straight line without information on which of the two half-lines the value is taken, we can obtain estimates comparable to (\ref{oomega}) due to Soundararajan \cite{sou} except for the imaginary axis:

\begin{Corollary}\label{Cor2}
Let $\phi\neq{\pi\over 2}$ and $\phi\in[0,\pi)$, then
$$
\max_{0<t_n(\phi)\leq T}\left|\zeta({\textstyle{\frac12}}+it_n(\phi))\right|\gg\exp\left(\left({1\over 2}+o(1)\right)\sqrt{\frac{\log T}{\log\log T}}\right).
$$
\end{Corollary}

\noindent The paper is organized as follows. In the next section we provide some preliminary results. In Section 3 we prove the key Proposition \ref{Prop1} which leads to both, Theorem \ref{Th2} and Corollary \ref{Cor2}. Section 4 contains the proof of Theorem \ref{Th1} and Corollary \ref{Cor1}.

\section{Preliminaries}

Recall the function $\Delta(s)$ defined in (\ref{feq}). It is well-known that
\begin{equation}\label{L2}
\Delta(\sigma+it)=\left({\vert t\vert\over 2\pi}\right)^{{\textstyle{1\over 2}}-\sigma-it}\exp(i(t+{\textstyle{\pi\over 4}}))(1+O(\vert t\vert^{-1}))\qquad\mbox{for}\quad \vert t\vert\geqslant 1
\end{equation}
uniformly for any $\sigma$ from a bounded interval. Hence,
\begin{equation}\label{star}
{1\over \Delta(s)-e^{2i\phi}}={-e^{-2i\phi}\over 1-e^{-2i\phi}\Delta(s)}=-e^{-2i\phi}\left(1+\sum_{k=1}^\infty e^{-2ki\phi}\Delta(s)^k\right)
\end{equation}
for $\sigma>{1\over 2}$. Obviously, $\Delta({1\over 2}+it)$ is a complex number from the unit circle for $t\in\R$. Moreover, $\Delta'({\textstyle{1\over 2}}+it)$ is non-vansishing for sufficiently large $t$ as follows from the asymptotic formula
\begin{equation}\label{delt}
{\Delta'\over \Delta}(\sigma+it)=-\log{\vert t\vert\over 2\pi}+O(\vert t\vert^{-1})\qquad\mbox{for}\quad \vert t\vert\geqslant 1.
\end{equation}

Next we introduce certain Dirichlet polynomials
\begin{equation}\label{DirPol}
X(s)=\sum_{n\leqslant X}\frac{x_n}{n^s},\quad Y(s)=\sum_{m\leqslant Y}\frac{y_m}{m^s},
\end{equation}
where $X, Y \leqslant T$. Moreover, we define the following quantities
\[
\mathcal{X}_{0}=\max_{n\leqslant X}|x_{n}|,\quad \mathcal{Y}_{0}=\max_{m\leqslant Y}|y_{m}|,\quad \mathcal{X}_{1} =
\sum\limits_{n\leqslant X}\frac{|x_{n}|}{n},\quad \mathcal{Y}_{1} =\sum\limits_{m\leqslant Y}\frac{|y_{m}|}{m}.
\]
and we set
\[
X_{1}(s) = \sum\limits_{n\leqslant X}\frac{\overline{x}_{n}}{n^{s}},\quad Y_{1}(s) =\sum\limits_{m\leqslant Y}\frac{\overline{y}_{m}}{m^{s}}.
\]

We shall use a variation of Lemma 5.1 from Ng \cite{Ng}:

\begin{Lemma}\label{NgLemma} 
Suppose the series $f(s) = \sum_{n =1}^{\infty}\alpha_{n}n^{-s}$ converges absolutely for $\Re s > 1$ and $\sum_{n = 1}^{\infty}|\alpha_{n}|n^{-\sigma}\ll (\sigma -1)^{-\,\gamma}$ for some $\gamma \geqslant 0$ as $\sigma\to 1+0$. Next, let $X(s)$ and $Y(s)$ be Dirichlet polynomials as defined in (\ref{DirPol}). Then, uniformly for  $a\in(1,2]$,
\begin{align*}
&J\,=\,\frac{1}{2\pi i}\int_{a+i}^{a+iT}f(s)X(s)Y(1-s)\frac{\Delta'}{\Delta}(s)\,ds\\
&\quad = -\frac{T}{2\pi}\biggl(\log \frac{T}{2\pi e}\biggr)\sum_{\substack{m\leqslant X\\ mn \leqslant Y}}\frac{\alpha_{n}x_{m}
y_{mn}}{mn}+O\left(\frac{Y^{a}(\log T)^{2}\mathcal{X}_{0}\mathcal{Y}_{0}}{(a-1)^{\gamma+1}}\right),
\end{align*}
where the implicit constant is absolute.
\end{Lemma}

\begin{proof}
Changing the order of summation and the integration, we get
\[
J = \sum\limits_{n = 1}^{\infty}\sum\limits_{m\leqslant X,\,k\leqslant Y}\biggl(\frac{k}{mn}\biggr)^{a}\frac{\alpha_{n}x_{m}y_{k}}{k}\,\frac{1}{2\pi}\int_{1}^{T}\frac{\Delta'}{\Delta}(a+it)\biggl(\frac{k}{mn}\biggr)^{it}dt.
\]
Next, the contribution of $O$ - term from (\ref{delt}) to $J$ does not exceed in order
\[
\mathcal{X}_{0}\mathcal{Y}_{0}(\log T)\sum\limits_{n =1}^{\infty}\frac{|\alpha_{n}|}{n^{a}}\sum\limits_{m\leqslant
X}\frac{1}{m^a}\sum\limits_{k\leqslant Y}k^{a-1}\ll \frac{Y^{a}(\log T)^{2}\mathcal{X}_{0}\mathcal{Y}_{0}}{(a-1)^{\gamma}}.
\]
Extracting the diagonal term (when $k = mn$) in the above expression for $J$, we get
\[
J =\biggl(-\int_{1}^{T}\log\frac{t}{2\pi}\frac{dt}{2\pi}\biggr)\sum_{\substack{m\leqslant X\\ mn \leqslant Y}}\frac{\alpha_{n}x_{m}
y_{mn}}{mn}+O(J_{1})+O\biggl(\frac{Y^{a}(\log T)^{2}\mathcal{X}_{0}\mathcal{Y}_{0}}{(a-1)^{\gamma}}\biggr),
\]
where
\begin{eqnarray*}
J_{1}&=& \sum\limits_{n =1}^{\infty}\sum\limits_{\substack{m\leqslant X,\,k\leqslant Y \\ mn\ne k}}\biggl(\frac{k}{mn}\biggr)^{a}\frac{|\alpha_{n}x_{m}y_{k}|}{k}\,|j_{k,mn}|\\ &\le& Y^{a-1}\mathcal{X}_{0}\mathcal{Y}_{0}\sum\limits_{n =1}^{\infty}\sum\limits_{\substack{m\leqslant X,\,k\leqslant Y \\ mn\ne k}}\frac{|\alpha_{n}|}{(mn)^{a}}\,|j_{k,mn}|
\end{eqnarray*}
with
$$
j_{k,r}=\int_{1}^{T}\biggl(\log\frac{t}{2\pi}\biggr)\biggl(\frac{k}{r}\biggr)^{it}\frac{dt}{2\pi}.
$$
Integrating by parts shows for $k\ne r$ that
\[
|j_{k,r}| =\biggl|\int_{1}^{T}\frac{\log(t/(2\pi))}{2\pi\log{(k/r)}}d\biggl(\frac{k}{r}\biggr)^{it}\biggr|\,\le\,\frac{2}{\pi}\log\frac{T}{2\pi}\biggl|\log\frac{k}{r}\biggr|^{-1}.
\]
Setting $r = mn$, $\beta_{r} = \sum_{n|r}|\alpha_{n}|$ in the expression for $J_{1}$, we have
\[
J_{1}\ll Y^{a-1}(\log T)\mathcal{X}_{0}\mathcal{Y}_{0}\sum\limits_{k\leqslant Y}\sum\limits_{r =1}^{\infty}\frac{\beta_{r}}{r^{a}}\biggl|\log\frac{k}{r}\biggr|^{-1}.
\]
Recall that $\zeta(s)$ has a simple pole at $s=1$. Thus, the contribution of the terms with $r\leqslant k/2$ and $r>3k/2$ does
not exceed in order
\[
Y^{a-1}(\log T)\mathcal{X}_{0}\mathcal{Y}_{0}\cdot Y\sum\limits_{r =1}^{\infty}\frac{\beta_{r}}{r^{a}}\,\ll\,Y^{a}\mathcal{X}_{0}\mathcal{Y}_{0}(\log T)\sum\limits_{n =1}^{\infty}\frac{|\alpha_{n}|}{n^{a}}\sum\limits_{m =1}^{\infty}\frac{1}{m^{a}}\,\ll\,\frac{Y^{a}(\log T)\mathcal{X}_{0}\mathcal{Y}_{0}}{(a-1)^{\gamma +1}}.
\]
For $k/2<r\leqslant 3k/2$, $r\ne k$ we set $r = k+\nu$; since $|\log(k/r)|^{-1}\ll k|\nu|^{-1}$, the corresponding part of $J_{1}$
can be estimated as follows:
\begin{eqnarray*}
\lefteqn{Y^{a-1}(\log T)\mathcal{X}_{0}\mathcal{Y}_{0}\sum\limits_{k\leqslant Y}\sum\limits_{0<|\nu|\leqslant k/2}\frac{\beta_{k+\nu}}{(k+\nu)^{a}}\,\frac{k}{|\nu|}}\\
&\ll& Y^{a}(\log T)\mathcal{X}_{0}\mathcal{Y}_{0}\sum\limits_{0<|\nu|\leqslant Y/2}\frac{1}{|\nu|}\sum\limits_{2|\nu|\leqslant k\leqslant Y}\frac{\beta_{k+\nu}}{(k+\nu)^{a}}\\
&\ll& Y^{a}(\log T)^{2}\mathcal{X}_{0}\mathcal{Y}_{0}\sum\limits_{k =1}^{\infty}\frac{\beta_{k}}{k^{a}}\ll \frac{Y^{a}(\log
T)^{2}\mathcal{X}_{0}\mathcal{Y}_{0}}{(a-1)^{\gamma+1}}.
\end{eqnarray*}
The lemma is proved.
\end{proof}

Our next lemma is a variation of Gonek's lemma:

\begin{Lemma}\label{gonekl}
Suppose the series $f(s) = \sum_{n=1}^{\infty}\alpha_{n}n^{-s}$ converges absolutely in the half-plane $\Re s > 1$, $\sum_{n =1}^{\infty}|\alpha_{n}|n^{-\sigma}\ll (\sigma - 1)^{-\,\gamma}$ for some $\gamma\geqslant 0$ as $\sigma\to 1+0$ and $\alpha_{n}\ll
n^{\varepsilon}$ for any $\varepsilon >0$. Then we have, for any fixed integer $m\geqslant 0$ and $c\geqslant 1$ uniformly for $a\in(1,2]$,
\begin{eqnarray*}
\lefteqn{\frac{1}{2\pi i}\int_{a+ic}^{a+iT}f(s)\Delta(1-s)\biggl(\frac{\Delta'}{\Delta}(s)\biggr)^{m}\,ds}\\
&=&(-1)^{m}\sum\limits_{n\leqslant \frac{T}{2\pi}}\alpha_{n}(\log n)^{m}\,+\,O\biggl(\frac{T^{a-{1\over 2}}}{(a-1)^{\gamma}}\,(\log T)^{m}+T^{{1\over 2}+\varepsilon}\biggr).
\end{eqnarray*}
\end{Lemma}

\noindent For the proof we refer to Lemma 5 from \cite{gonekk2} (in the original paper the remainder term is not uniform in $a>1$).
\par

Finally, we consider the divisor function $r\mapsto d(r):=\sum_{d\mid n}1$ and some of its generalizations, respectively. For any positive real $\kappa$ and $\Re s > 1$ define the numbers $d_\kappa(n)$ by 
\[
\zeta(s)^{\kappa} = \sum_{n=1}^{\infty}d_\kappa(n)n^{-s}.
\]
Notice that $n\mapsto d_\kappa(n)$ is a multiplicative function, on a prime power $p^j$ given by
\[
d_\kappa(p^{\,j}) =\frac{\Gamma(\kappa+j) }{\Gamma(\kappa)j!}
\]
(see \cite[formula (5)]{HBrown}). We continue with two lemmas on this generalized divisor function:

\begin{Lemma}
Let $\kappa$ be a positive real number and $n$ a positive integer.
\begin{enumerate}
\item For $\kappa\geq 0$ and $n\geq 1$ we have $d_\kappa(n)\geq 0$.
\item For fixed $n$, $d_\kappa(n)$ increases with respect to $\kappa$.
\item For fixed $\kappa\geq 0$ and $\epsilon>0$ we have $d_\kappa(n)\ll n^{\epsilon}$.
\item If $j$ is an integer then
\[
d_{\kappa j}(n)=\sum_{n=n_1n_2\ldots n_j}d_\kappa(n_1)d_\kappa(n_2)\ldots d_\kappa(n_j).
\]
\end{enumerate}
\end{Lemma}

\noindent For a proof we refer to \cite{HBrown}.

\begin{Lemma}\label{gendivprob}
For any fixed positive real $\lambda, \mu$ we have, as $x\to\infty$,
\[
\sum\limits_{n\le x}d_{\lambda}(n)d_{\mu}(n)n^{-1}\,\asymp\,_{\lambda,\,\mu}(\log x)^{\lambda\mu},
\]
and, if $\lambda\geq 2$ is a positive integer, 
$$
\sum_{n\leq x}d_\lambda(n)=xP_{\lambda-1}(\log x)+O(x^{\delta_\lambda}), 
$$
where $P_{\lambda-1}$ is a computable polynomial of degree $\lambda-1$ and $\delta_\lambda$ is a positive quantity strictly less than one.
\end{Lemma}

\noindent These assertions can be established by the standard technique based on Perron's formula and contour integration (see \cite{ivic}, Chapter 13).

\section{The key proposition}

In order to prove Theorem \ref{Th1} and Corollary \ref{Cor2} we consider the discrete moments
\begin{equation}\label{s1}
S_1(T)=\sum_{0<t_{n}(\phi)\leqslant T} \zeta\left({\textstyle \frac12 }-it_{n}(\phi)\right)X\left({\textstyle \frac12}+it_{n}(\phi)\right)Y\left({\textstyle \frac12}-it_{n}(\phi)\right)
\end{equation}
and
\begin{equation}\label{s2}
S_2(T)=\sum_{0<t_{n}(\phi)\leqslant T} \left |X\left({\textstyle \frac12 }+it_{n}(\phi)\right)\right|^{2}.
\end{equation}
with Dirichlet polynomials $X(s)$ and $Y(s)$ which will be specified later. Our first aim is the following

\begin{Proposition}\label{Prop1}
Let $X(s)$ and $Y(s)$ be Dirichlet polynomials as defined in (\ref{DirPol}). Then
for any $\phi\in[0,\pi)$, as $T\rightarrow\infty$,
\begin{eqnarray}\label{prop1}
\lefteqn{}
S_1(T)&=&\frac{T}{2\pi}\biggl(\log\frac{T}{2\pi e}\biggr)\biggl(e^{-2i\phi}\sum\limits_{\substack{m\leqslant X\\ mn\leqslant Y}}\frac{x_{m}y_{mn}}{mn}\,+\,\sum\limits_{\substack{m\leqslant Y\\ mn\leqslant X}}\frac{y_{m}x_{mn}}{mn}\biggr)\,+\,O(R_{1}),
\end{eqnarray}
where
$$
R_{1}\,=\,(X+Y)\sqrt{T}(\log T)^{2}\mathcal{X}_{1}\mathcal{Y}_{1} +(X+Y)(\log T)^{4}\mathcal{X}_{0}\mathcal{Y}_{0};
$$
moreover,
\begin{equation}\label{prop2}
S_2(T)=\frac{T}{2\pi}\biggl(\log\frac{T}{2\pi
e}\biggr)\sum\limits_{n\leqslant X}\frac{|x_{n}|^{2}}{n}\,+\,O(R_{2}),
\end{equation}
where
$$
R_{2}\,=\,X\sqrt{T}(\log T)^{2}\sum\limits_{n\leqslant X}\frac{|x_{n}|^{2}}{n}+X(\log T)^3\mathcal{X}_{0}^2.
$$
All implicit constants are absolute.
\end{Proposition}

\subsection{Proof of Proposition \ref{Prop1}}

\subsubsection{Proof of (\ref{prop1})}\label{square}

We begin with the estimations
\begin{align*}
& \bigl|\zeta\bigl(\tfrac{1}{2}+it\bigr)\bigr|\,\ll\,T^{1/6}(\log T)^{3/2},\\
& \bigl|X\bigl(\tfrac{1}{2}+it\bigr)\bigr|\,\le\,\sum\limits_{n\leqslant
X}\frac{|x_{n}|}{\sqrt{n}}\,=\,\sum\limits_{n\leqslant
X}\sqrt{n}\,\frac{|x_{n}|}{n}\,\le\,\sqrt{X}\mathcal{X}_{1},\quad
\bigl|Y\bigl(\tfrac{1}{2}+it\bigr)\bigr|\,\le\,\sqrt{Y}\mathcal{Y}_{1};
\end{align*}
the first one is a well-known bound from zeta-function theory (see \cite{ivic}), whereas the second and third one are straightforward.
Hence, it is sufficient to obtain (\ref{prop1}) for the sum over the segment $c<t_{n}(\phi)\leqslant T$, where $c>32\pi$ is a large
absolute constant ($32\pi$ comes from the inequality $2(\frac{t}{2\pi})^{-\frac12}<\frac12$ that is used in the proof).

Next, without loss of generality we may set $T ={1\over 2}(t_{\nu}(\phi)+t_{\nu+1}(\phi))$. Indeed, otherwise we may replace $T$ by the closest value $T_{1}$ of such type. Then the error of such replacement in the right-hand side of (\ref{prop1})
is bounded by
\begin{eqnarray*}
\lefteqn{\biggl(\log\frac{T}{2\pi}\biggr)^{-1}\biggl(\log\frac{T}{2\pi}\biggr)\biggl(\sum\limits_{\substack{m\leqslant X
\\ mn\leqslant Y}}\frac{|x_{m}y_{mn}|}{mn}\,+\,\sum\limits_{\substack{m\leqslant Y\\ mn\leqslant X}}\frac{|y_{m}x_{mn}|}{mn}\biggr)}\\
&\ll &\mathcal{X}_{0}\mathcal{Y}_{0}\sum\limits_{mn\leqslant T}\frac{1}{mn}\,\ll\,\mathcal{X}_{0}\mathcal{Y}_{0}\sum\limits_{r\leqslant T}\frac{d(r)}{r}\,\ll\,\mathcal{X}_{0}\mathcal{Y}_{0}(\log T)^{2},
\end{eqnarray*}
where we have used Lemma \ref{gendivprob} and the asymptotics $T\log T$ for the number of $t_n(\phi)\leqslant T$ (see Theorem 1 from \cite{Kalp}). Since the points $s =\tfrac{1}{2}+it_{n}(\phi)$ are the roots of the function
$\Delta(s)-e^{2i\phi}$, the sum in question can be rewritten as a contour integral:
\begin{eqnarray*}
\lefteqn{\sum_{\substack{c< t_{n}(\phi)\leqslant T}}\zeta(\textstyle{\frac12}-it_{n}(\phi))X(\textstyle{\frac12}+it_{n}(\phi))Y(\textstyle{\frac12}-it_{n}(\phi))}\\
&=&\frac{1}{2\pi i }\int_{\square} \zeta(1-s)X(s)Y(1-s)\frac{\Delta'(s)}{\Delta(s)-e^{2i\phi}}\,ds;
\end{eqnarray*}
here $\square$ stands for the counterclockwise oriented rectangular contour with vertices $a+ic,\ a+iT,\ 1-a+iT,\ 1-a+ic$, where
$a=1+(\log T)^{-1}$. Let $\mathcal{I}_1$ and $\mathcal{I}_3$ be integrals over right and left sides of contour, and $\mathcal{I}_2$
and $\mathcal{I}_4$ be the integrals over the top and bottom edges of the contour. We may assume the constant $c$ so large that the
relations
\[
|\Delta(a+it)|\,=\,\biggl(\frac{t}{2\pi}\biggr)^{1/2-a}\bigl(1\,+\,O(t^{-1})\bigr)\,\le\,2\biggl(\frac{t}{2\pi}\biggr)^{-1/2}\,<\,\frac{1}{2}
\]
hold for any $t>c$.  In view of (\ref{star}) we have 
\begin{eqnarray*}
\mathcal{I}_1&=&\frac{1}{2\pi i}\int_{a+ic}^{a+iT}\zeta(1-s)X(s)Y(1-s)\,\frac{\Delta'(s)}{\Delta(s)-e^{2i\phi}}\,ds\\
&=&-\,\frac{e^{-2i\phi}}{2\pi i}\int_{a+ic}^{a+iT}\frac{\zeta(s)}{\Delta(s)}\,X(s)Y(1-s)\Delta'(s)\biggl(1\,+\,\sum\limits_{k=1}^{\infty}e^{-2ik\phi}\Delta(s)^{k}\biggr)ds\,\\
&=&\,-e^{-2i\phi}(j_{1}\,+\,j_{2}),
\end{eqnarray*}
where
\begin{align*}
& j_{1}\,=\,\frac{1}{2\pi
i}\int_{a+ic}^{a+iT}\zeta(s)X(s)Y(1-s)\frac{\Delta'}{\Delta}(s)ds,\\
& j_{2}\,=\,\frac{1}{2\pi
i}\int_{a+ic}^{a+iT}\zeta(s)X(s)Y(1-s)\frac{\Delta'}{\Delta}(s)\sum\limits_{k
= 1}^{\infty}e^{-2ik\phi}\Delta(s)^{k}\,ds.
\end{align*}
We observe for $s=a+it$
\begin{align*}
& |X(a+it)|\,\le\,\sum\limits_{n\leqslant
X}\frac{|x_{n}|}{n^{a}}\,\leqslant \,
\mathcal{X}_{1}, \quad |Y(1-a-it)|\,\le\,\sum\limits_{m\leqslant Y}\frac{m^{a}|y_{m}|}{m}\,\ll\, Y\mathcal{Y}_{1},\\
& \biggl|\sum\limits_{k=
1}^{\infty}e^{-2ik\phi}\Delta(a+it)^{k}\biggr|\,\le\,2\biggl(\frac{t}{2\pi}\biggr)^{-1/2}\sum\limits_{k
= 0}^{\infty}\frac{1}{2^{k}}\,\ll\,t^{-1/2}.
\end{align*}
Thus, we have
\[
|j_{2}|\,\ll\,\zeta(a)Y\mathcal{X}_{1}\mathcal{Y}_{1}\int_{c}^{T}\frac{\log
{t}\,dt}{\sqrt{t}}\,\ll\,Y\sqrt{T}(\log
T)^{2}\mathcal{X}_{1}\mathcal{Y}_{1}.
\]
Applying Lemma \ref{NgLemma}  to $j_{1}$ we get
\begin{multline*}
\mathcal{I}_{1}\,=\,e^{-2i\phi}\frac{T}{2\pi}\biggl(\log\frac{T}{2\pi
e}\biggr)\sum\limits_{\substack{m\leqslant X \\
mn\leqslant Y}}\frac{x_{m}y_{mn}}{mn}\,+ O\bigl(Y\sqrt{T}(\log
T)^{2}\mathcal{X}_{1}\mathcal{Y}_{1}+Y(\log
T)^{4}\mathcal{X}_{0}\mathcal{Y}_{0}\bigr).
\end{multline*}

In a similar way we may compute ${\mathcal I}_3$. We observe
\[
\mathcal{I}_{3}\,=\,-\frac{1}{2\pi}\int_{c}^{T}\zeta(a-it)X(1-a+it)Y(a-it)\,\frac{\Delta'(1-a+it)}{\Delta(1-a+it)-e^{2i\phi}}\,dt.
\]
This in combination with $\overline{X}(s) = X_{1}(\overline{s}), \overline{Y}(s) = Y_{1}(\overline{s})$ yields
\[
\overline{\mathcal{I}}_{3}\,=\,-\frac{1}{2\pi
i}\int_{a+ic}^{a+iT}\zeta(s)X_{1}(1-s)Y_{1}(s)\,\frac{\Delta'(1-s)}{\Delta(1-s)-e^{-2i\phi}}\,ds.
\]
In view of (\ref{star}) we find $\overline{\mathcal{I}}_{3} = -(j_{3}+j_{4})$, where the expressions for $j_{3}$ and $j_{4}$ can be obtained by the replacing $X(s)$ with $Y_{1}(s)$ and $Y(1-s)$ with $X_{1}(1-s)$ in the expressions for $j_{1}$ and $j_{2}$. Applying Lemma 4 to $j_{3}$ and estimating $j_{4}$ similarly to $j_{2}$, we get
\begin{multline*}
\overline{\mathcal{I}}_{3}\,=\,\frac{T}{2\pi}\biggl(\log\frac{T}{2\pi
e}\biggr)\sum\limits_{\substack{m\leqslant Y \\
mn\leqslant
X}}\frac{\overline{y}_{m}\overline{x}_{mn}}{mn}+O\bigl(X\sqrt{T}(\log
T)^{2}\mathcal{X}_{1}\mathcal{Y}_{1}+X(\log
T)^{4}\mathcal{X}_{0}\mathcal{Y}_{0}\bigr).
\end{multline*}
In order to estimate $\mathcal{I}_{2}$ we first note that the following inequalities hold along the line segment of integration:
\begin{align*}
& |\zeta(1-s)|\,\ll\, \sqrt{T}\log T,\quad
|X(s)|\,\le\,\sum\limits_{n\leqslant X}\frac{|x_{n}|}{n}\,n^{1-\sigma}\,\ll\,X^{1-\sigma}\mathcal{X}_{1},\\
& |Y(1-s)|\,\le\,\sum\limits_{n\leqslant
Y}\frac{|y_{n}|}{n}\,n^{\sigma}\,\ll\,Y^{\sigma}\mathcal{Y}_{1},
\end{align*}
and, finally,
\begin{eqnarray*}
|\zeta(1-s)X(s)Y(1-s)|&\ll&\sqrt{T}\mathcal{X}_{1}\mathcal{Y}_{1}\,X\biggl(\frac{Y}{X}\biggr)^{\sigma}\log T\\
&\ll& X\sqrt{T}\mathcal{X}_{1}\mathcal{Y}_{1}\biggl\{\biggl(\frac{Y}{X}\biggr)^{1-a}\,+\,\biggl(\frac{Y}{X}\biggr)^{a}\biggr\}\log T\\
&\ll& (X+Y)\sqrt{T}(\log T)\mathcal{X}_{1}\mathcal{Y}_{1}.
\end{eqnarray*}
Next, by (\ref{delt}) we get
\[
\frac{\Delta'(s)}{\Delta(s)-e^{2i\phi}}\,=\,\frac{\Delta'(s)}{\Delta(s)}\biggl(1\,+\,\frac{e^{2i\phi}}{\Delta(s)-e^{2i\phi}}\biggr)\,\ll\,(\log
T)\biggl(1\,+\,\frac{1}{|\Delta(s)-e^{2i\phi}|}\biggr).
\]
The second term in the brackets is bounded by an absolute constant. Indeed, in the case $\sigma\geqslant \tfrac{1}{2} +
\tfrac{1}{3}\bigl(\log\frac{T}{2\pi}\bigr)^{-1}$ by (\ref{L2}) we have
\[
|\Delta(\sigma +
iT)|\,=\,\biggl(\frac{T}{2\pi}\biggr)^{1/2-\sigma}\bigl(1\,+\,O(T^{-1})\bigr)\,\le\,e^{-1/3}\bigl(1\,+\,O(T^{-1})\bigr)<\frac{1}{2},
\]
and hence $|\Delta(s)-e^{2i\phi}|\geqslant 1 -
|\Delta(s)|>\tfrac{1}{2}$. Similarly, in the case $\sigma\leqslant
\tfrac{1}{2} - \tfrac{1}{3}\bigl(\log\frac{T}{2\pi}\bigr)^{-1}$ we
have
\[
|\Delta(\sigma +
iT)|\,\ge\,e^{1/3}\bigl(1\,+\,O(T^{-1})\bigr)\,>\,\frac{4}{3},\quad
|\Delta(s)-e^{2i\phi}|>\frac{4}{3}-1 = \frac{1}{3}.
\]
Finally, let
\[
\frac{1}{2}\,-\,\frac{1}{3}\biggl(\log\frac{T}{2\pi}\biggr)^{-1}\,<\,\sigma\,<\,\frac{1}{2}\,+\,\frac{1}{3}\biggl(\log\frac{T}{2\pi}\biggr)^{-1}.
\]
Then, using the relations
\[
\Delta\bigl(\tfrac{1}{2}+iT\bigr)\,=\,e^{-2i\vartheta(T)}, \quad
\Delta(\sigma+iT)\,=\,\tau\,e^{-2i\vartheta(T)}\bigl(1\,+\,O(T^{-1})\bigr),
\]
where $\tau = \bigl(T/(2\pi)\bigr)^{1/2-\sigma}$ and $\vartheta =
\vartheta(T)$  denotes the increment of any fixed continuous branch
of the argument of $\pi^{-s/2}\Gamma\bigl(s/2\bigr)$ along the line
segment with end-points $s = \tfrac{1}{2}$ and $s =
\tfrac{1}{2}+iT$ (check \cite{Edwards}), we have $e^{-1/3}\leqslant \tau\leqslant e^{1/3}$
and
\begin{eqnarray*}
\Delta(\sigma +iT)-e^{2i\phi}&=&\bigl(\Delta(\sigma+iT)-\Delta\bigl(\tfrac{1}{2}+iT\bigr)\bigr)\,+\,\bigl(\Delta\bigl(\tfrac{1}{2}+iT\bigr)-e^{2i\phi}\bigr)\\
&=&(\tau -1)e^{-2i\vartheta}\,-\,2ie^{i(\phi-\vartheta)}\sin{(\phi+\vartheta)}\,+\,O(T^{-1})\\
&=& e^{-i\vartheta}\bigl((\tau
-1)\cos\vartheta\,+2\sin{(\vartheta+\phi)}\sin{\phi}-\\
& &-
\,i((\tau-1)\sin\vartheta+2\sin{(\vartheta+\phi)\cos\phi)}\bigr)\,+\,O(T^{-1}).
\end{eqnarray*}
Thus we obtain
\begin{eqnarray*}
\bigl|\Delta(\sigma+iT)-e^{2i\phi}\bigr|^{2}&=&(\tau-1)^{2}\,+\,4\tau\sin^{2}{(\vartheta+\phi)}\,+\,O(T^{-1})\\
&\geqslant &\,4\tau\sin^{2}(\vartheta+\phi)\,+\,O(T^{-1}).
\end{eqnarray*}
Using the fact that $T ={1\over 2}(t_{\nu}(\phi)+t_{\nu+1}(\phi))$ for some $\nu$, we finally get
\[
\sin^{2}(\vartheta+\phi)\,=\,\sin^{2}\biggl(\pi\nu+\frac{\pi}{2}+O(T^{-1})\biggr)\,\ge\,\sin^{2}\frac{\pi}{3}\,=\,\frac{3}{4}
\]
and hence, for sufficiently large $T$,
\[
|\Delta(\sigma+iT)-e^{2i\phi}|^{2}\,\ge\,4\cdot\frac{3}{4}\,e^{-1/3}\,+\,O(T^{-1})\,>\,2.
\]
Thus, $|\Delta(s)-e^{2i\phi}|>\tfrac{1}{3}$ for any $s$ under consideration, hence
\[
\mathcal{I}_{2}\,\ll\,(X+Y)\sqrt{T}(\log T)^{2}\mathcal{X}_{1}\mathcal{Y}_{1}.
\]
The integral $\mathcal{I}_{4}$ can be estimated in a similar way. The relation (\ref{prop1}) is proved.

\subsubsection{Proof of (\ref{prop2})}

In view of the inequalities
\[
\bigl|X\bigl(\tfrac{1}{2}+it_{n}(\phi)\bigr)\bigr|^2\,\le\,\biggl(\sum\limits_{n\leqslant
X}\frac{|x_{n}|}{\sqrt{n}}\biggr)^{2}\,\le\,X\sum\limits_{n\leqslant X}\frac{|x_{n}|^{2}}{n}
\]
it suffices to consider only the sum over the segment $c<t_{n}(\phi)\leqslant T$. Next, we may set $T ={1\over 2}(t_{\nu}(\phi)+t_{\nu +1}(\phi))$. Then we have
\begin{eqnarray*}
\sum\limits_{c<t_{n}(\phi)\leqslant T}\bigl|X\bigl(\tfrac{1}{2}+it_{n}(\phi)\bigr)\bigr|^2&=&\sum\limits_{c<t_{n}(\phi)\leqslant T}X\bigl(\tfrac{1}{2}+it_{n}(\phi)\bigr)X_{1}\bigl(\tfrac{1}{2}-it_{n}(\phi)\bigr)\\
&=&{1\over 2\pi i}\int_{\square}X(s)X_{1}(1-s)\frac{\Delta'(s)}{\Delta(s)-e^{2i\phi}}\,ds,
\end{eqnarray*}
where $\square$ stands for the rectangular contour defined in Section \ref{square}.
Denoting the integrals $\mathcal{I}_{k}$, $1\leqslant k\leqslant 4$ as above, we get
\begin{eqnarray*}
\mathcal{I}_{1}&=&\frac{1}{2\pi i}\int_{a+ic}^{a+iT}X(s)X_{1}(1-s)\frac{\Delta'(s)}{\Delta(s)-e^{2i\phi}}ds\\
&=&-\frac{1}{2\pi i}\int_{a+ic}^{a+iT}X(s)X_{1}(1-s)\frac{\Delta'}{\Delta}(s)\,\sum\limits_{k
= 1}^{\infty}e^{-2ik\phi}\Delta(s)^{k}\,ds.
\end{eqnarray*}

Estimating the integrand as in Section \ref{square} we find
\[
\mathcal{I}_{1}\,\ll\,X\mathcal{X}_{1}^{2}\int_{c}^{T}\frac{\log
t\,dt}{\sqrt{t}}\,\ll\,X\sqrt{T}(\log
T)\mathcal{X}_{1}^{2}\,\ll\,X\sqrt{T}(\log
T)^{2}\sum\limits_{n\leqslant X}\frac{|x_{n}|^{2}}{n}.
\]
Furthermore,
\[
\mathcal{I}_{3}\,=\,-\frac{1}{2\pi}\int_{c}^{T}X(1-a+it)X_{1}(a-it)\frac{\Delta'(1-a+it)}{\Delta(1-a+it)-e^{2i\phi}}dt,
\]
and the relations $\overline{X}(s) = X_{1}(\overline{s})$, $\overline{X}_{1}(s) = X(\overline{s})$ imply
\begin{eqnarray*}
\overline{\mathcal{I}}_{3}&=&-\frac{1}{2\pi i}\int_{c}^{T}X_{1}(1-s)X(s)\frac{\Delta'(1-s)}{\Delta(1-s)-e^{-2i\phi}}ds\\
&=&-\frac{1}{2\pi i}\int_{a+ic}^{a+iT}X(s)X_{1}(1-s)\frac{\Delta'}{\Delta}(s)\,\frac{ds}{1-e^{-2i\phi}\Delta(s)}\,=\,-j_{1}+\mathcal{I}_{1},
\end{eqnarray*}
where
\[
j_{1}\,=\,\frac{1}{2\pi i}\int_{a+ic}^{a+iT}X(s)X_{1}(1-s)\frac{\Delta'}{\Delta}(s)\,ds.
\]
Lemma \ref{NgLemma} with $f(s)\equiv 1$ (that is, $\alpha_{1} =1$, $\alpha_{n} = 0$ for $n>1$), $\gamma = 0$ and $Y(s) = X_{1}(s)$, applied to $j_{1}$ yields
\[
j_{1}\,=\,-\frac{T}{2\pi}\biggl(\log\frac{T}{2\pi
e}\biggr)\sum\limits_{n\leqslant
X}\frac{|x_{n}|^{2}}{n}\,+\,O\bigl(X(\log
T)^{3}\mathcal{X}_{0}^{2}\bigr).
\]
Using the above bound for $\mathcal{I}_{1}$, we derive
\[
\mathcal{I}_{3}\,=\,\frac{T}{2\pi}\biggl(\log\frac{T}{2\pi
e}\biggr)\sum\limits_{n\leqslant
X}\frac{|x_{n}|^{2}}{n}\,+\,O\biggl(X\sqrt{T}(\log
T)^{2}\sum\limits_{n\leqslant
X}\frac{|x_{n}|^{2}}{n}\biggr)\,+\,O\bigl(X(\log
T)^{3}\mathcal{X}_{0}^{2}\bigr).
\]
Estimating $\mathcal{I}_{2}$ and taking into account the bounds
\[
X(s)\,\ll\,X^{1-\sigma}\mathcal{X}_{1},\quad
X_{1}(1-s)\,\ll\,X^{\sigma}\mathcal{X}_{1},\quad
\frac{\Delta'(s)}{\Delta(s)-e^{2i\phi}}\,\ll\,\log T
\]
for $s = \sigma + iT$, $1-a\le\sigma\leqslant a$, $T = {1\over 2}(t_{\nu}(\phi)+t_{\nu+1}(\phi))$, we get
\[
\mathcal{I}_{2}\,\ll\,X(\log T)\mathcal{X}_{1}^{2}\,\ll\,X(\log
T)^{2}\sum\limits_{n\leqslant X}\frac{|x_{n}|^{2}}{n}.
\]
The integral $\mathcal{I}_{4}$ can be estimated in a similar way. Thus, formula (\ref{prop2}) is proved.

\subsection{Proof of Theorem \ref{Th1}}

Suppose that $k = p/q$, where $p> q\geqslant 1$, $(p,q) = 1$. Let $\kappa = 1/q$, $r = p-q$, $\xi =T^{1/(4p)}$.  We define the coefficients for the polynomials in (\ref{DirPol}) by
\[
X(s)\,=\,\biggl(\sum\limits_{n\leqslant \xi}\frac{d_{\kappa}(n)}{n^{s}}\biggr)^{p}\,=\,\sum\limits_{n\leqslant
\xi^{p}}\frac{d_{\kappa p}(n;\xi)}{n^{s}},\;\;Y(s)\,=\,\biggl(\sum\limits_{n\leqslant \xi}\frac{d_{\kappa}(n)}{n^{s}}\biggr)^{r}\,=\,\sum\limits_{n\leqslant \xi^{r}}\frac{d_{\kappa r}(n;\xi)}{n^{s}},
\]
where $d_{\kappa m}(n;\xi)$ is given by 
\[
d_{\kappa m}(n;\xi)\,=\,\sum\limits_{\substack{n = n_{1}\cdots n_{m}
\\ n_{1},\ldots, n_{m}\leqslant \xi}}d_{\kappa}(n_{1})\ldots d_{\kappa}(n_{m}).
\]
By Lemma 7 it is easy to see that $d_{\kappa m}(n;\xi) = d_{\kappa m}(n)$ for $m\leqslant \xi$ and $0\leqslant d_{\kappa m}(n,\;\xi)\leqslant d_{\kappa m}(n)$ for $m>\xi$. Next, we consider the moment  (\ref{s1}). 
%\[
%S_{1}(T)\,=\,\sum\limits_{0<t_{n}(\phi)\leqslant
%T}\zeta\bigl(\tfrac{1}{2}-it_{n}(\phi)\bigr)X\bigl(\tfrac{1}{2}+it_{n}(\phi)\bigr)Y\bigl(\tfrac{1}{2}-it_{n}(\phi)\bigr).
%\]
By (\ref{prop1}),
\[
S_{1}(T)\,=\,\frac{T}{2\pi}\biggl(\log\frac{T}{2\pi
e}\biggr)\bigl(e^{-2i\phi}\Sigma_{1}\,+\,\Sigma_{2}\bigr)\,+\,O(R_{1}),
\]
where
\begin{eqnarray*}
\Sigma_{1}=\sum\limits_{m\leqslant \xi^{p},\,mn\leqslant
\xi^{r}}\frac{d_{\kappa p}(m;\xi)d_{\kappa r}(mn;\xi)}{mn} &\le&
\sum\limits_{n\le\xi^{r}}\frac{d_{\kappa
r}(n)}{n}\sum\limits_{l|n}d_{\kappa p}(l)\\
&=& \sum\limits_{n\leqslant \xi^{r}}\frac{d_{\kappa r}(n)d_{\kappa
p+1}(n)}{n}\,\ll\,(\log T)^{\kappa r(\kappa p+1)},
\end{eqnarray*}
\begin{eqnarray*}
\Sigma_{2}=\sum\limits_{m\leqslant \xi^{r},\,mn\leqslant
\xi^{p}}\frac{d_{\kappa r}(m;\xi)d_{\kappa
p}(mn;\xi)}{mn}&\ge&\sum\limits_{n\leqslant \xi}\frac{d_{\kappa
p}(n)}{n}\sum\limits_{l|n}d_{\kappa r}(l)\\
&=&\sum\limits_{n\leqslant \xi}\frac{d_{\kappa p}(n)d_{\kappa r
+1}(n)}{n}\,\gg\,(\log T)^{\kappa p(\kappa r+1)},
\end{eqnarray*}
and
\begin{eqnarray*}
R_{1}&\ll&(\xi^{p}+\xi^{r})\sqrt{T}(\log
T)^{2}\sum\limits_{n\leqslant \xi^{p}}\frac{d_{\kappa
p}(n;\xi)}{n}\sum\limits_{m\leqslant \xi^{r}}\frac{d_{\kappa r}(m;\xi)}{m}\\
&\ll& T^{3/4}(\log T)^{2+\kappa(p+r)}\,\ll T^{4/5}.
\end{eqnarray*}
Thus,
\[
|S_{1}(T)|\,\ge\,\frac{T}{2\pi}\biggl(\log\frac{T}{2\pi
e}\biggr)\bigl(\Sigma_{2}-\Sigma_{1}\bigr)+O\bigl(T^{4/5}\bigr)\gg\,T(\log
T)^{1+\kappa p(\kappa r+1)}\,\gg\, T(\log T)^{k^{2}+1}.
\]
On the contrary, using H\"{o}lder's inequality we get 
\begin{eqnarray*}
|S_{1}(T)| &\le& \biggl(\sum\limits_{0<t_{n}(\phi)\leqslant T}\bigl|\zeta\bigl(\tfrac{1}{2}+it_{n}(\phi)\bigr)\bigr|^{2k}\biggr)^{1/(2k)}\times\\
&&\times\,\biggl(\sum\limits_{0<t_{n}(\phi)\leqslant T}\bigl|X\bigl(\tfrac{1}{2}+it_{n}(\phi)\bigr)\bigr|^{2k/(2k-1)}\cdot\bigl|Y\bigl(\tfrac{1}{2}+it_{n}(\phi)\bigr)\bigr|^{2k/(2k-1)}\biggr)^{1-1/(2k)}\\
&=& \biggl(\sum\limits_{0<t_{n}(\phi)\leqslant
T}\bigl|\zeta\bigl(\tfrac{1}{2}+it_{n}(\phi)\bigr)\bigr|^{2k}\biggr)^{1/(2k)}\bigl(S_{2}(T)\bigr)^{1-1/(2k)},
\end{eqnarray*}
where $S_2(T)$ is given by (\ref{s2}).  
In view of (\ref{prop2}) we find that
\[
|S_{2}(T)|\,=\,\frac{T}{2\pi}\biggl(\log\frac{T}{2\pi
e}\biggr)\sum\limits_{n\leqslant \xi^{p}}\frac{d^{2}_{\kappa
p}(n;\xi)}{n}\,+\,O\bigl(\xi^{p}\sqrt{T}(\log
T)^{k^{2}+1}\bigr)\,\ll\, T(\log T)^{k^{2}+1}.
\]
Hence,
\[
\sum\limits_{0<t_{n}(\phi)\leqslant
T}\bigl|\zeta\bigl(\tfrac{1}{2}+it_{n}(\phi)\bigr)\bigr|^{2k}\,\ge\,\frac{\bigl(S_{1}(T)\bigr)^{2k}}{\bigl(S_{2}(T)\bigr)^{2k-1}}\,\gg\,
T(\log T)^{k^{2}+1}.
\]
Theorem \ref{Th1} is proved.

\subsection{Proof of Corollary \ref{Cor2}}

Our argument follows Soundararajan \cite{sou}. Taking $X=Y$, resp. $x_n=y_n$ in (\ref{DirPol}), we get 
$$
S_1(T)=\sum_{0<t_{n}(\phi)\leqslant T} \zeta\left({\textstyle \frac12 }-it_{n}(\phi)\right)\vert X\left({\textstyle \frac12}+it_{n}(\phi)\right)\vert^2
$$
for (\ref{s1}). Comparing with (\ref{s2}) we find
\begin{equation}\label{schoko}
|S_1(T)|\leq S_2(T)\max_{0<t_n\leq T}\vert \zeta({\textstyle{1\over 2}}+it_n(\phi))\vert .
\end{equation}
Now let $L=\exp(\sqrt{\log X\log\log X})$ where $X$ is a sufficiently large parameter which will be chosen later. Following Soundararajan \cite{sou}, we define $x_n=n^{1\over 2}f(n)$, where $f$ is the multiplicative function such that $f(p^k)=0$ for all primes $p$ and positive integers $k\geq 2$,
$$
f(p)={L\over \sqrt{p}\log p}
$$
for all primes $p$ satisfying $L^2\leq p\leq \exp((\log L)^2)$, and $f(p)=0$ for all other primes. We observe that
$$
{\mathcal X}_0=\max_{n\leq X}\sqrt{n}f(n)\leq L^m\prod_{j=1}^m{1\over \log p_j}, 
$$
where $p_1,\ldots,p_m$ are the least distinct $m$ prime numbers in $[L^2,\exp((\log L)^2]$ for which $n=p_1\cdot\ldots\cdot p_m\leq X$. 
Since $X\ge n\ge L^{2m}$ then $L^{m}\le X^{\frac12}$ and $\mathcal{X}_{0} <
L^{m}\le X^{\frac12}$.
Moreover, since $f(n)\leq 1$ for any $n$, we find
$$
{\mathcal X}_1=\sum_{n\leq X}{f(n)\over \sqrt{n}}
\leq \sum_{n\leq X}{1\over \sqrt{n}}
\ll X^{1\over 2}
$$   
as well as
\begin{align*}
\mathcal{X}_{2}=&\sum\limits_{n\le
X}\frac{|x_{n}|^{2}}{n}=\sum\limits_{n\le
X}f^{2}(n)=\sum\limits_{\substack{n = p_{1}\ldots p_{m}\le X
\\ L^{2}<p_{1}\ldots p_{m}\le e^{L^{2\mathstrut}}}}\frac{L^{2m}}{(p_{1}\log{p_{1}}\ldots
p_{m}\log{p_{m}})^{2}}\\
&\le\prod\limits_{L^{2}<p\le
e^{L^{2\mathstrut}}}\left(1+\frac{L^{2}}{p^{2}\log^{2}p}\right)<\exp{\biggl(L^{2}\sum\limits_{p>L^{2}}\frac{1}{p^{2}\log^{2}p}\biggr)}<e.
\end{align*}
Inserting these bounds in the asymptotic formulae of Proposition \ref{Prop1} yields
$$
S_1(T)=(1+e^{-2i\phi})\frac{T}{2\pi}\biggl(\log\frac{T}{2\pi e}\biggr)\sum\limits_{\substack{mn\leqslant X}}{f(m)f(mn)\over \sqrt{n}}
+O(X^2T^{1\over 2}(\log T)^2)
$$
and
$$
S_2(T)=\frac{T}{2\pi}\biggl(\log\frac{T}{2\pi e}\biggr)\sum\limits_{n\leqslant X}|f(n)|^{2}+O(X T^{1\over 2}(\log T)^2+X^2(\log T)^3).
$$
Let $X=T^{{1\over 4}-\epsilon}$, then the main terms in the latter formulae dominate the error terms and we may deduce from (\ref{schoko}) that 
$$
\max_{0<t_n\leq T}\vert \zeta({\textstyle{1\over 2}}+it_n(\phi))\vert\geq {|S_1(T)|\over S_2(T)}\gg \left(\sum\limits_{\substack{mn\leqslant X}}{f(m)f(mn)\over \sqrt{n}}\right)\left( \sum\limits_{n\leqslant X}|f(n)|^{2}\right)^{-1}.
$$ 
Soundararajan \cite{sou} proved that the right hand-side is $\geq \exp\left((1+o(1))\sqrt{{\log X\over \log\log X}}\right)$ which gives the desired estimate by letting $\epsilon \to 0$. 

\section{The third moment}

Corollary \ref{Cor2} shows that large values of the zeta-function appear on all straight lines through the origin with the imaginary axis as only possible exception. More subtle information on the value-distribution with respect to half-lines can be derived from the third discrete moment. Our first aim is the asymptotic formula of Theorem \ref{Th2}, and we conclude with the proof of Corollary \ref{Cor1}. 

\subsection{Proof of Theorem \ref{Th2}}

The method of proof is along the lines of Kalpokas and Steuding \cite{Kalp}. It suffices to evaluate 
\[
S(T)\,=\,\sum\limits_{c<t_{n}(\phi)\leqslant
T}\zeta\bigl(\tfrac{1}{2}+it_{n}(\phi)\bigr)\zeta^{2}\bigl(\tfrac{1}{2}-it_{n}(\phi)\bigr),
\]
where $c > 32\pi$ is an absolute constant and $T =\tfrac{1}{2}(t_{\nu}(\phi)+t_{\nu + 1}(\phi))$ for some $\nu$. Setting $a=1+(\log T)^{-1}$, we find by Cauchy's theorem
\begin{eqnarray*}
S(T)&=&\frac{1}{2\pi i}\left\{\int_{a+ic}^{a+iT}+\int_{a+iT}^{1-a+iT}+\int_{1-a+iT}^{1-a+ic}+\int_{1-a+ic}^{a+ic}\right\}\zeta(s)\zeta(1-s)^2\frac{\Delta'(s)}{\Delta(s)-e^{2i\phi}}\,ds\\
&=&\sum_{k=1}^{4}\mathcal{I}_k, 
\end{eqnarray*}
say. 

First we consider $\mathcal{I}_1$. In view of (\ref{star}) we obtain similar to the analogous case of $\mathcal{I}_1$ in Section \ref{square}
\begin{eqnarray*}
\mathcal{I}_1&=&
%-\frac{e^{-2i\phi}}{2\pi i}\int_{a+ic}^{a+iT}\zeta^{3}(s)\Delta(1-s)\frac{\Delta'(s)}{\Delta(s)}\frac{ds}{1-e^{-2i\phi}\Delta(s)}\,=\\
-\frac{e^{-2i\phi}}{2\pi i}\int_{a+ic}^{a+iT}\zeta^{3}(s)\Delta(1-s)\frac{\Delta'}{\Delta}(s)\biggl(1+\sum\limits_{k
=1}^{\infty}e^{-2ik\phi}\Delta^{k}(s)\biggr)ds\\
&=&-e^{-2i\phi}\bigl(j_{1}+e^{-2i\phi}j_{2}+j_{3}\bigr),
\end{eqnarray*}
where
\begin{align*}
j_{1}\,&=\,\frac{1}{2\pi i}\int_{a+ic}^{a+iT}\zeta^{3}(s)\Delta(1-s)\frac{\Delta'}{\Delta}(s)ds,\\
j_{2}\,&=\,\frac{1}{2\pi i}\int_{a+ic}^{a+iT}\zeta^{3}(s)\frac{\Delta'}{\Delta}(s)\,ds,\\
j_{3}\,&=\,\frac{1}{2\pi i}\int_{a+ic}^{a+iT}\zeta^{3}(s)\frac{\Delta'}{\Delta}(s)\sum\limits_{k
= 1}^{\infty}e^{-2i(k+1)\phi}\Delta^{k}(s)ds.
\end{align*}

By Gonek's Lemma \ref{gonekl} (with $m = 1$) and Lemma \ref{gendivprob} we have
\[
j_{1}\,=\,-\sum\limits_{n\le\frac{T}{2\pi}}d_{3}(n)\log n\,+\,O\bigl(T^{{1\over 2}+\varepsilon}\bigr)\,=\,-\frac{T}{2\pi}P_{3}\biggl(\log\frac{T}{2\pi}\biggr)\,+\,O\bigl(T^{{1\over 2}+\varepsilon}\bigr),
\]
where $P_3(x)$ is a computable polynomial of degree three.
 
 Next, Lemma \ref{NgLemma} with $X(s)\equiv 1$, $Y(s)\equiv 1$, applied to $j_{2}$, leads to
\[
j_{2}\,=\,-\frac{T}{2\pi}\biggl(\log\frac{T}{2\pi e}\biggr)\,+\,O\bigl((\log T)^{6}\bigr).
\]
Finally, by standard arguments we obtain
\[
j_{3}\,\ll\,\int_{c}^{T}\frac{\log t}{(a-1)^{3}}\frac{dt}{\sqrt{t}}\,\ll\,\sqrt{T}(\log T)^{4}.
\]
Hence,
$$
\mathcal{I}_1=e^{-2i\phi}\frac{T}{2\pi}\,P_{3}\biggl(\log\frac{T}{2\pi}\biggr)\,+\,e^{-4i\phi}\frac{T}{2\pi}\biggl(\log\frac{T}{2\pi
e}\biggr)\,+\,O\bigl(T^{{1\over 2}+\varepsilon}\bigr).
$$

Further, transforming the integral $\mathcal{I}_{3}$ we find
\begin{eqnarray*}
\overline{\mathcal{I}_{3}}&=&-\frac{1}{2\pi i}\int_{a+ic}^{a+iT}\zeta(1-s)\zeta^{2}(s)\frac{\Delta'(1-s)\,ds}{\Delta(1-s)-e^{-2i\phi}}\\
&=&-\frac{1}{2\pi i}\int_{a+ic}^{a+iT}\zeta^{3}(s)\Delta(1-s)\frac{\Delta'}{\Delta}(s)
\sum_{k=0}^{\infty}e^{-2i\phi}\Delta(s)^k ds.
\end{eqnarray*}
The latter expression equals $e^{2i\phi}\,\mathcal{I}_{1}$, hence we may deduce (do not forget to conjugate)
$$
\mathcal{I}_{3}=\frac{T}{2\pi}P_{3}\biggl(\log\frac{T}{2\pi}\biggr)\,+\,e^{2i\phi}\frac{T}{2\pi}\biggl(\log\frac{T}{2\pi e}\biggr)\,+\,O\bigl(T^{{1\over 2}+\varepsilon}\bigr).
$$

In order to estimate the integral $\mathcal{I}_{2}$ over the top and bottom edges we write
\begin{eqnarray*}
F(s)&=&\zeta^{3}(s)\Delta(1-s)\,\frac{\Delta'}{\Delta}(s)\,\frac{1}{\Delta(s)-e^{2i\phi\mathstrut}}\\
&=&\zeta^{3}(1-s)\Delta(s)\frac{\Delta'}{\Delta}(s)\biggl(1\,+\,\frac{e^{2i\phi}}{\Delta(s)-e^{2i\phi}}\biggr).
\end{eqnarray*}
Since $T = \tfrac{1}{2}(t_{\nu}(\phi)+t_{\nu+1}(\phi))$ for some $\nu$, the inequality $|\Delta(s)-e^{2i\phi}|>\tfrac{1}{3}$ from Paragraph 3.1.1 holds over the segment of integration. Using the bound $|\zeta(\sigma+it)|\ll t^{(1-\sigma)/3}\log t$, for $s = \sigma +iT$, $\tfrac{1}{2}\leqslant \sigma\leqslant a$ (see \cite{ivic}), we get
\[
|F(s)|\,\ll\,(\log T)|\zeta^{3}(s)\Delta(1-s)|\,\ll\,(\log T)\bigl(T^{\frac{1}{3\mathstrut}(1-\sigma)}\bigr)^{3}T^{\sigma -
\frac{1}{2\mathstrut}}\,\ll\,\sqrt{T}(\log T)^{4}.
\]
In the case $1-a\leqslant \sigma\leqslant \tfrac{1}{2}$ we have
\[
|F(s)|\,\ll\,(\log T)|\zeta^{3}(1-s)\Delta(s)|\,\ll\,(\log T)\bigl(T^{\frac{1}{3}(1-(1-\sigma))}\log
T\bigr)^{3}T^{\frac{1}{2}-\sigma}\,\ll\,\sqrt{T}(\log T)^{4}.
\]
Thus, $\mathcal{I}_{2}\ll \sqrt{T}(\log T)^{4}$. Finally, the bound $\mathcal{I}_{4} = O(1)$ is obvious. Collecting together the above
results, we obtain
\[
S(T)\,=\,(1+e^{-2i\phi})\frac{T}{2\pi}P_{3}\biggl(\log\frac{T}{2\pi}\biggr)\,+\,(e^{2i\phi}+e^{-4i\phi})\frac{T}{2\pi}\biggl(\log\frac{T}{2\pi e}\biggr)\,+\,O\bigl(T^{{1\over 2}+\varepsilon}\bigr)
\]
Now it remains to note that we must multiply $S(T)$ by $e^{4i\phi}$ to obtain
\[
\sum\limits_{0<t_{n}(\phi)\leqslant T}\zeta^{3}\bigl(\tfrac{1}{2}+it_{n}(\phi)\bigr)=e^{4i\phi}S(T)\,+\,O(1).
\]
The theorem is proved.
\medskip

\textsc{Remark.} It is possible to compute the coefficients of the polynomial $P_{3}$ as follows. Define the polynomial $P_{2}(u) =
A_{2}u^{2}+A_{1}u+A_{0}$ by the relation
\[
\sum\limits_{n\leqslant x}d_{3}(n)\,=\,xP_{2}(\log x)\,+\,o(x),
\]
which is a special case of the asymptotics from Lemma \ref{gendivprob} We get
\[
P_{2}(\log x)\,=\,\text{res}_{s =1}\biggl(\frac{x^{s}\zeta^{3}(s)}{s}\biggr)
\]
and hence
$
A_{2} = \tfrac{1}{2},\quad A_{1}\,=\,3\gamma - 1,\quad A_{0}\,=\,1+3(\gamma^{2}-\gamma+\gamma_{1}),
$
where $\gamma, \gamma_{1},\ldots$ are the coefficients of Laurent expansion
\[
\zeta(s)\,=\,\frac{1}{s-1}\,+\,\gamma\,+\,\gamma_{1}(s-1)\,+\ldots.
\]
Thus, using the definition of $P_{3}(u)$ and Abel's summation formula, we find 
\[
P_{3}(u)\,=\,uP_{2}(u)-P_{2}(u)+P_{2}^{'}(u)-P_{2}^{''}(u)\,=\,\sum\limits_{k= 0}^{3}B_{k}u^{k},
\]
where $B_{3}=A_{2}=\tfrac{1}{2}$, $B_{2}=A_{1}-A_{2}=3\gamma - \tfrac{3}{2}$, 
$B_{1}=A_{0}-A_{1}+2A_{2}=3\bigl(\gamma_{1}+(1-\gamma)^{2}\bigr)$ and 
$B_{0}=-B_{1}=-3\bigl(\gamma_{1}+(1-\gamma)^{2}\bigr)$. For the values of the coefficients $\gamma_{j}$ and $P_{2}$ we refer to \cite{Israilov}.

\subsection{Proof of Corollary \ref{Cor1}}

First we observe for any non-negative integer $\ell$
\begin{eqnarray*}
\lefteqn{\sum_{t_n^{\pm}(\phi)\leqslant T}\left|\zeta(\textstyle{\frac12}+it_n^{\pm}(\phi))\right|^{2\ell+1}}\\
&=& \frac12\sum_{t_n(\phi)\leqslant T}\left(\left|\zeta(\textstyle{\frac12}+it_n(\phi))\right|^{2\ell+1}\pm e^{-(2\ell+1)i\phi}\zeta\left(\textstyle{\frac12}+it_n(\phi)\right)^{2\ell+1}\right)
\end{eqnarray*}
(with the same choice of signs on either side of the equation). In view of Theorem \ref{Th2} and Theorem \ref{Th1} with $k={3\over 2}$ we get
$$
\sum_{t_n^{\pm}(\phi)\leqslant T}\left|\zeta(\textstyle{\frac12}+it_n^{\pm}(\phi))\right|^3 \gg T(\log T)^{13\over 4}.
$$
Since the number of intersection points $t_n(\phi)\leqslant T$ is bounded by $T\log T$ (see Theorem 1 from \cite{Kalp}) and 
\[
\sum_{0<t_n^\phi\leqslant T}\left |\zeta\left({\textstyle{\frac12}}+it_n(\phi)\right)\right |^2\ll T(\log T)^2
\]
(see Theorem 2 from  \cite{Kalp}), we find
\begin{align*}
\sum_{t_n^{\pm}(\phi)\leqslant T}\left|\zeta(\textstyle{\frac12}+it_n^{\pm}(\phi))\right|^3  & \ll 
\max_{t_n^{\pm}(\phi)\leqslant T} \left|\zeta(\textstyle{\frac12}+it_n^{\pm}(\phi))\right|
\sum_{0<t_n^\phi\leqslant T}\left|\zeta\left({\textstyle{\frac12}}+it_n(\phi)\right)\right|^2\\
&\ll \max_{t_n^{\pm}(\phi)\leqslant T} \left|\zeta(\textstyle{\frac12}+it_n^{\pm}(\phi))\right|
T(\log T)^2.
\end{align*}
Comparing both estimates we arrive at
$$
\max_{t_n^{\pm}(\phi)\leqslant T}\left|\zeta(\textstyle{\frac12}+it_n^{\pm}(\phi))\right|\gg (\log T)^{\frac54},
$$

If we assume Riemann Hypothesis we can use the following estimate (see Theorem 1 from \cite{Christ}) that holds for any non-negative real $k$, uniformly for $\phi\in[0,\pi)$, as $T\to\infty$,
\[
\sum_{0<t_n^\phi\leqslant T}\left |\zeta\left({\textstyle{\frac12}}+it_n(\phi)\right)\right |^{2k}\ll T(\log T)^{k^2+1+\epsilon}.
\]
We have 
\begin{align*}
T(\log T)^{13\over 4}\ll \sum_{t_n^{\pm}(\phi)\leqslant T}\left|\zeta(\textstyle{\frac12}+it_n^{\pm}(\phi))\right|^3  & \ll 
\max_{t_n^{\pm}(\phi)\leqslant T} \left|\zeta(\textstyle{\frac12}+it_n^{\pm}(\phi))\right|^{\alpha}
\sum_{0<t_n^\phi\leqslant T}\left|\zeta\left({\textstyle{\frac12}}+it_n(\phi)\right)\right|^{3-\alpha}\\
&\ll \max_{t_n^{\pm}(\phi)\leqslant T} \left|\zeta(\textstyle{\frac12}+it_n^{\pm}(\phi))\right|^{\alpha}
T(\log T)^{(\frac32-\frac{\alpha}{2})^2+1+\epsilon}.
\end{align*}
Comparing both sides we arrive at
$$
\max_{t_n^{\pm}(\phi)\leqslant T}\left|\zeta(\textstyle{\frac12}+it_n^{\pm}(\phi))\right|\gg (\log T)^{\frac32-\frac{\alpha}{4}-\frac{\epsilon}{\alpha}}.
$$
After choosing $\alpha = 2\delta$ and $\epsilon=\delta^2$, where $\delta$ is positive and arbitrary small we get
$$
\max_{t_n^{\pm}(\phi)\leqslant T}\left|\zeta(\textstyle{\frac12}+it_n^{\pm}(\phi))\right|\gg (\log T)^{\frac32-\delta}.
$$

which proves the corollary.

{\bf Acknowledgements: } The authors would like to thank prof. I.D.Shkredov for a suggestion that helped to improve the lower bound of Corollary \ref{Cor1}.

\bigskip

\ \bigskip

\tiny

\noindent\parbox{7cm}{
Justas Kalpokas\\
Faculty of Mathematics and Informatics\\
Vilnius University\\
Naugarduko 24, 03225 Vilnius, Lithuania\\
justas.kalpokas@mif.vu.lt}
\medskip

\noindent
J\"orn Steuding\\
Department of Mathematics, W\"urzburg University\\
Am Hubland, 97\,218 W\"urzburg, Germany\\
steuding@mathematik.uni-wuerzburg.de
\medskip

\noindent
Maxim A. Korolev\\
Steklov Mathematical Institute\\
Gubkina str. 8, 119991, Moscow, Russia\\
hardy\_ramanujan@mail.ru, korolevma@mi.ras.ru


\small

\begin{thebibliography}{11}

%\bibitem{bohr}{\sc H. Bohr, R. Courant}, Neue Anwendungen der Theorie der diophantischen Approximationen auf die Riemannsche Zetafunktion, {\it J. reine angew. Math.} {\bf 144} (1914), 249-274 {\bf do we need this reference? and others?}

 \bibitem{Christ}{\sc T. Christ, J. Kalpokas}, Upper bounds for discrete moments of the derivatives of the Riemann zeta-function on the critical line, {\it Lithuanian Math. J.}, accepted

\bibitem{Edwards}{\sc H.M. Edwards}, {\it Riemann's zeta function}, Academic Press 1974

\bibitem{HBrown}{\sc D. R. Heath-Brown}, Fractional Moments of the Riemann Zeta-Function {\it J. London Math. Soc.} {\bf 2-24} (1981) 65-78.

%\bibitem{conre}{\sc J.B. Conrey}, More than two fifths of the zeros of the Riemann zeta-function are on the critical line, {\it J. reine angew. Math.} {\bf 399} (1989), 1-26

%\bibitem{cgg}{\sc J.B. Conrey, A. Ghosh, S.M. Gonek}, Simple zeros of the zeta function of a quadratic number field. I, {\it Invent. Math.} {\bf 86} (1986), 563-576

%\bibitem{gonekk}{\sc J.B. Conrey, A. Ghosh,  S.M. Gonek}, Simple zeros of the Riemann zeta-function, {\it Proc. London Math. Soc.} {\bf 76(3)} (1998), 497-522

%\bibitem{ds}{\sc R.D. Dixon, L. Schoenfeld}, The size of the Riemann zeta-function at places symmetric with respect to the point $1/2$, {\it Duke Math. J.} {\bf  33} (1966), 291-292

%\bibitem{edwa}{\sc H.M. Edwards}, {\it Riemann's zeta function}, Academic Press 1974

%\bibitem{fuj}{\sc A. Fujii}, Gram's law for the zeta zeros and the eigenvalues of Gaussian unitary ensembles, {\it Proc. Japan Acad.} {\bf 63} (1987), 392-395

%\bibitem{gs}{\sc R. Garunk\v stis, J. Steuding}, On the roots of the equation $\zeta(s)=a$, submitted

\bibitem{gonekk2}{\sc S.M. Gonek}, Mean values of the Riemann zeta-function and its derivatives, {\it Invent. Math.} {\bf 75}:1 (1984), 123-141

\bibitem{gram}{\sc J. Gram}, Sur les z\'eros de la fonction $\zeta(s)$ de Riemann, {\it Acta Math.} {\bf 27} (1903), 289-304

%\bibitem{hutch}{\sc J.I. Hutchinson}, On the roots of the Riemann zeta function, {\it Trans. A.M.S.} {\bf 27} (1925), 49-60

\bibitem{ivic}{\sc A. Ivi\'c}, {\it The Riemann zeta-function}, John Wiley \& Sons, New York 1985

%\bibitem{joyner}{\sc D. Joyner}, {\it Distribution theorems of $L$-functions}, Pitman Research Notes in Mathematics, 1986

\bibitem{Kalp} {\sc J. Kalpokas, J. Steuding} On the Value-Distribution of the Riemann Zeta-Function on the Critical Line, {\it Moscow Jour. Combinatorics and Number Theo.} {\bf 1} (2011), 26-42

%\bibitem{levi}{\sc N. Levinson}, More than one third of Riemann's zeta-function are on $\sigma={1\over 2}$, {\it Adv. Math.} {\bf 13} (1974), 383-436

\bibitem{Mili} {\sc M.B. Milinovich and N. Ng} Lower bound for the moments of $\zeta'(\rho)$, preprint, available at arXiv:0706.2321v1 

%\bibitem{moser}{\sc J. Moser}, The proof of the Titchmarsh hypothesis in the theory of the Riemann zeta-function, {\it Acta Arith.} {\bf 36} (1980), 147-156 (Russian)

\bibitem{Ng}{\sc N. Ng}, A discrete mean value of the derivative of the Riemann zeta function, {\it Mathematika} {\bf 54} (2007), 113-155

\bibitem{Sound}{\sc Z. Rudnick and K. Soundararajan}, Lower bounds for moments of L-functions, {\it Proc. Natl.
Sci. Acad. USA} {\bf 102} (2005), 6837-6838.

\bibitem{sou}{\sc K. Soundararajan}, Extreme values of zeta and L-functions, {\it Math. Ann.} {\bf 342} (2008), 467-486

%\bibitem{spira}{\sc R. Spira}, An inequality for the Riemann zeta function, {\it Duke Math. J.} {\bf 32} (1965), 247-250

%\bibitem{titch}{\sc E.C. Titchmarsh}, On van der Corput's method and the zeta-function of Riemann, IV, {\it Quart. J. Math.} {\bf 5} (1934), 98-105

%\bibitem{trudg}{\sc T. Trudgian}, Gram's law fails a positive proportion of the time, submitted to the arXiv {\bf isn't this published in Acta Arithmetica???}

\bibitem{Israilov} {\sc M.I.Israilov}, On the Laurent expansion of the Riemann
zeta-function, {\it Proc. Steklov Inst. Math.} {\bf 158} (1983),
105-112

\end{thebibliography}
\end{document}